\documentclass[12pt,reqno]{amsart}
\usepackage{graphicx, cite} 
\usepackage{amssymb,amsthm,,amsfonts,amsmath,mathtools}
\usepackage{xcolor}

\newcommand\blfootnote[1]{%
	\begingroup
	\renewcommand\thefootnote{}\footnote{#1}%
	\addtocounter{footnote}{-1}%
	\endgroup
}

\def\dj{d\kern-0.4em\char"16\kern-0.1em}
\def\Dj{\mbox{\raise0.3ex\hbox{-}\kern-0.4em D}}

\def\be{\begin{equation}}
\def\ee{\end{equation}}

 \newcommand{\beq}{\begin{equation}}
 \newcommand{\eeq}{\end{equation}}

 \newcommand{\bea}{\begin{eqnarray}}
 \newcommand{\eea}{\end{eqnarray}}

 \newcommand{\beas}{\begin{eqnarray*}}
 \newcommand{\eeas}{\end{eqnarray*}}

 \newcommand{\beqs}{\begin{equation*}}
 \newcommand{\eeqs}{\end{equation*}}

 \newcommand{\bi}{\begin{itemize}}
 \newcommand{\ei}{\end{itemize}}

 \newcommand{\ben}{\begin{enumerate}}
 \newcommand{\een}{\end{enumerate}}

 \newcommand{\ba}{\begin{array}}
 \newcommand{\ea}{\end{array}}

 \newcommand{\R}{\mathbb R}
 \newcommand{\N}{\mathbb N}
 
 \newcommand{\Z}{\mathbb Z}

\newcommand{\NN}{\mathbb N}

\newcommand{\CC}{\mathbb C}
\newcommand{\RR}{\mathbb R}
\newcommand{\ZZ}{\mathbb Z}
\newcommand{\ds}{\displaystyle}

\newcommand{\SSS}{\mathcal S}

\newcommand{\VV}{\mathcal V}

\def\R{\mathbb{R}}
\def\N{\mathbb{N}}
\def\Z{\mathbb{Z}}

\def\dss{\displaystyle}

 \def\Z{\mathbf{Z}}
\def\N{\mathbb{N}}

\newtheorem{te}{Theorem}[section]
\newtheorem{lema}{Lemma}[section]
\newtheorem{prop}{Proposition}[section]

\newtheorem{rem}{Remark}[section]
\newtheorem{ex}{Example}[section]

\definecolor{darkred}{rgb}{0.8,0,0}

\begin{document}


\title{Almost diagonalization of $\Psi$DO's over various generalized function spaces
}
\author[Stevan Pilipovi\'c]{Stevan Pilipovi\'c$^*$}\blfootnote{$^*$ corresponding author}

\address{\hspace{-\parindent} Stevan Pilipovi\'c, Serbian Academy of Sciences and Arts, Serbia}

\email{stevan.pilipovic@gmail.com}

\author[Nenad Teofanov]{Nenad Teofanov}

\address{\hspace{-\parindent} Nenad Teofanov, Department of Mathematics and Informatics, Faculty of Sciences
	University of Novi Sad, Serbia}

\email{nenad.teofanov@dmi.uns.ac.rs}

\author[Filip Tomi\'c]{Filip Tomi\'c}

\address{\hspace{-\parindent} Filip Tomi\'c, Faculty of Technical Sciences,
	University of Novi Sad, Serbia}

\email{filip.tomic@uns.ac.rs}



\maketitle

\begin{abstract}
Inductive and projective type sequence spaces of sub- and super-exponential growth, and the corresponding inductive and projective limits of modulation spaces are  considered as a framework for almost diagonalization of pseudo-differential operators. Moreover, recent results of the first author and B. Prangoski related to the almost diagonalization of pseudo-differential operators in the context of H\"ormander metrics are reviewed.
\end{abstract}

\section{Introduction}

The main goal of this paper is to  offer a brief review of some recent results on 
 the almost diagonalization of pseudo-differential operators with symbols in various projective and inductive limits of modulation spaces, and spaces of generalized functions. The results were presented at the conference in Sarajevo dedicated to the 75. anniversary of academician Mirjana Vukovi\'c.

\par

Properties of pseudo-differential operators depend on the assigned classes of symbols. Here we  consider the Weyl correspondence between operators and symbols, see \eqref{Weyl-PsiDO}. Apart from the classical H\"ormander classes (\cite{horm3}), certain modulation spaces are recognized to be useful symbol classes, see \cite{Gr2} where the tools of time-frequency analysis are used in approximate diagonalization of related operators. This approach is thereafter developed and successfully used in different contexts, see \cite{CR} and the references given there. Let us just mention sparse decompositions for Schr\"odinger-type propagators given in \cite{CNR}, and diagonalization in the framework of tempered ultra-distributions, \cite{pt}.

General results from \cite{gro-rz} are recently extended to H\"ormander metrics by the first author and B. Prangoski in \cite{PPHM}.
It turns out that the class of weights used in \cite{PPHM} could be extended to the class of moderate weights (see subsection \ref{weights}). The main aim of this paper is to provide necessary background material for investigations in that direction.
This includes the introduction of new symbol classes as well as exposition of results for approximate diagonalization in the context of Gelfand-Shilov spaces.

We end this section with an explanation of the idea behind the notion of approximate diagonalization of operators.

%
%
%

\par

\subsection{Motivation}
Let us start with a general and simple example of a matrix type operator on a Hilbert space.
 Let $\psi_n, n\in\N$ be a basis for a separable Hilbert space $\mathcal H$,
 $f=\sum_{n\in\N} a_n\psi_n\in\mathcal H$, and $A:\mathcal H\rightarrow \mathcal H$ be linear and continuous.
Then
$$Af=\sum_{n\in\N}a_nA\psi_n=\sum_{n\in\N}a_n\sum_{m\in\N}b_{n,m}\psi_m=\sum_{m\in\N}\sum_{n\in\N}b_{n,m}a_n\psi_m.
$$
So, $A$ can be viewed as the action of an infinite matrix on a space of sequences:
$$
(a_n)_{n\in \N}\to (Af)_{m\in \N} = \big(\sum_{n\in\N} b_{n,m}a_n\big)_{m\in \N},
$$
More generally, instead of $\N \times \N$ one can observe indices in $\Lambda$, a discrete  subgroup of $\R^{2d}$. Such group is often represented  as $A \,\mathbb{Z}^{2d}$ ($\mathbb Z^{2d}$ is the set of integer points in $\mathbb R^{2d}$), where $A$ is a $2d$-dimensional, regular matrix with determinant ${\rm det} A<1$. We will also use the term lattice
for such $\Lambda$.

For the sake of simplicity, consider the lattice points of the form
$$
\lambda=(\alpha k,\beta i) \in \Lambda, \qquad k, i\in\mathbb Z^d.
$$
i.e.  $\Lambda=\alpha\mathbb Z^d\times\beta\mathbb Z^d$.
Then the time-frequency shifts of $g\in L^2(\mathbb R^d)$ are given by
$$\pi(\lambda) g=
\pi_{\alpha k, \beta i} g= e^{2\pi\alpha k\cdot t}g(t-\beta i), \;\; \lambda=(\alpha k,\beta i), \;\;
k \cdot t=\langle k,t \rangle =  \sum_j^dk_jt_j.
$$

Recall that the set $\mathcal{G}(g,\Lambda)=\{\pi(\lambda)g;\lambda \in\Lambda\}$ is  a Gabor frame in
$L^2(\RR^d)$  if for  every $f\in L^2(\RR^d)$
there exist $c_1,c_2>0$ such that
\begin{equation}\label{fr}
c_1\sum_{\lambda\in\Lambda}|\langle f,\pi(\lambda)g\rangle|^2\leq||f||_{L^2}^2\leq\
c_2\sum_{\lambda\in\Lambda}|\langle f,\pi(\lambda)g\rangle|^2
\end{equation}
($\langle \cdot ,\cdot \rangle$ here denotes the scalar product in $L^2(\mathbb R^d)$). If $c_1=c_2$ then $\mathcal{G}(g,\Lambda)$ is called a tight frame and $\gamma=g$.
If  $\mathcal{G}(g,\Lambda)$ is a frame, then there exists a dual window  $\gamma\in L^2(\RR^d)$ such that
\begin{equation}\label{frexpansion}
f=\sum_{\lambda\in\Lambda}\langle f,\pi(\lambda)g\rangle\pi(\lambda)\gamma \in L^2(\RR^d), \qquad 
f \in L^2(\RR^d).
\end{equation}

 If 
$H_0(x)=e^{-a|x|^2}$, $ x\in\RR^d,$ $ a>0$, and $ \lambda\in \Lambda =\alpha\mathbb Z^d\times\beta\mathbb Z^d$, $\alpha \beta <1$, then $\mathcal G(H_0,\lambda)$
is a frame in ${L^2}(\RR^d)$, and its dual frame $\gamma$ satisfies
$$|\gamma(x)|+|\widehat{\gamma}(x)|\leq Ce^{-c|x|^2}, \quad x\in\RR^d.
$$
Such frames are also referred as superframes. For the details we refer to \cite{GLy, YuL, SW}.
Moreover, it is well known (\cite{CR}Theorem 3.2.21.) that for  $g(x)=e^{-\pi x^2}$, $x\in \R^d$,the set
$${\mathcal G}(g, \Lambda)={\mathcal G}(g,\alpha, \beta) = \{\pi(\lambda)g;\lambda \in\Lambda\}$$ 
is a Gabor frame for $L^2 (\R^d)$ if and only if 
$\alpha \beta <1$. 
We also mention that if $\alpha>0$ and $g\in W(\R^d)$ are such that
$$a\leq \sum_{k\in \Z^d}|g(x-\alpha k)|^2\leq b, \quad x\in \R^d,$$
then there exists $\beta_0=\beta_0 (\alpha)$ such that ${\mathcal G}(g,\alpha, \beta)$ is the frame for all $\beta \leq \beta_0$. Here, $W(\R^d)$ is the Wiener space which consists of function that are locally bounded, and globally in $L^1(\R^d)$. For details we refer to \cite{CR} (see Section 3).

\par

Next we introduce pseudo-differential operators. 
Let $\pi(\lambda)g$, $\lambda\in\Lambda\subset \mathbb R^{2d}$ be a tight frame, and for $f,\phi\in L^2(\mathbb R^d)$
we consider expansions
$$
f(t)=\sum_{\lambda\in\Lambda}a_{\lambda}\pi(\lambda)g(t), \qquad
\phi(t)=\sum_{\lambda\in\Lambda}b_{\lambda}\pi(\lambda)g(t), \quad t \in \mathbb R^d.
$$

The Weyl-H\"ormander pseudodifferential operator (or the Weyl transform) $a^w$ with the symbol $a \in \mathcal{S}^{\prime}\left(\mathbb{R}^{2 d}\right)$ (see subsection \ref{notation} for the notation) is defined by
\begin{equation} \label{Weyl-PsiDO}
a^w f(x)=\int_{\mathbb{R}^{2d}} a\left(\frac{x+y}{2}, \xi\right) e^{2 \pi i(x-y) \cdot \xi} f(y) d y d \xi, \quad {f\in \mathcal S(\R^d)}.
\end{equation}
Then we have
\begin{multline*}
\langle a^w(t,D)f(t),\phi(t)\rangle \\ 
= \langle \sum_{(k,i) \in \mathbb Z^{2d}}a_{k,j} a^w(t,D)\pi_{\alpha k,\beta i}g(t),\sum_{(p,q)\in\mathbb Z^{2d}}b_{p,q}\pi_{\alpha p,\beta q}g(t)\rangle \\
= \sum_{k,i}\sum_{(p,q)\in\mathbb Z}a_{k,i}b_{p,q}\langle a^w(t,D) \pi_{\alpha k,\beta i}g(t),\pi_{\alpha p,\beta q}g(t) \rangle.
\end{multline*}

If we denote by $A=A_{\ZZ^{2d}\times\ZZ^{2d}}$ the infinite dimensional matrix of the dimension $\mathbb Z^{2d}\times\mathbb Z^{2d}$  with elements 
$$
\langle a^w(t,D) \pi_{\alpha k,\beta i}g(t),\pi_{\alpha p,\beta q}g(t)\rangle, \qquad (k,i), (p,q) \in\mathbb Z^{2d},
$$
then
$$\langle a^w(t,D)f(t),\phi(t)\rangle=(a_{k,i})_{1\times\ZZ^{2d}}\,A\,(b_{p,q})_{\ZZ^{2d}\times 1},\qquad  (k,i),(p,q)\in\ZZ^{2d},
$$

where the expansion of $f$ and $\phi$ is clear.

In such a way we may represent $a^w$ as a matrix type operator whose properties are determined by
the matrix elements. Asymptotic decay estimates of these elements away from diagonal are related to mapping properties of the corresponding operator. For that reason the term approximate diagonalization  is used to describe techniques
based on these observations.

Thus, the goal is to characterize the decrease of the matrix $A$ far from its diagonal. In the continuous case, instead of a matrix, we use an integral representation, namely 
the short-time Fourier transform (STFT)  as it will be explained in Section \ref{preliminaries}.


\section{Preliminaries} \label{preliminaries}

In this section we fix general notation, and then proceed with basic facts on weight functions, short-time Fourier transform, modulation spaces, and pseudo-differential operators.

\subsection{Notation} \label{notation}
The sets of all all integers, positive integers, nonnegative integers, real and complex numbers are denoted by $\ZZ, \mathbb{N}, \mathbb{N}_0, \mathbb{R}$ and $\mathbb{C}$, respectively. For $x=\left(x_1, \ldots, x_d\right) \in \mathbb{R}^d$, and a multi-index $\alpha= \left(\alpha_1, \ldots, \alpha_d\right) \in \mathbb{N}_0^d$, we use the notation: $ |x|:=\left(x_1^2+\ldots+x_d^2\right)^{1 / 2} $,  $x^\alpha:=\prod_{j=1}^d x_j^{\alpha_j}$,  
$ |\alpha|:=\alpha_1+\ldots+\alpha_d $, $ \alpha !:= \alpha_{1} ! \cdots \alpha_{d} ! $, $D^\alpha=D_x^\alpha:=D_1^{\alpha_1} \cdots D_d^{\alpha_d}$, where $D_j^{\alpha_j}:=\left(-i \partial / \partial x_j\right)^{\alpha_j}(j=1, \ldots, d)$. We write capital letters $X,Y,Z...$ for elements in $ \mathbb{R}^{2 d}$, and $ A \lesssim B $ means
$ A \leq cB$ for a suitable constant $c>0$.

The symbol $K \subset \subset V$ for an open $V \subset \mathbb{R}^d$ means that $K$ is a compact subset of $V$. By $\hookrightarrow$ we denote continuous embeddings between two Banach spaces. The norm in $L^p\left(\mathbb{R}^d\right)$ is denoted by $\|\cdot\|_p, \, 1\leq p \leq \infty$, and the corresponding sequence spaces will be denoted by $l^p$. Fourier transform $\mathcal F$ of a function $f\in L^1$ will be denoted by 
$$\dss {\hat f}(\xi)= \int_{\R^d}f(x) e^{-2\pi i x\xi}dx \qquad (\mathcal F^{-1}f(\xi)=\mathcal Ff(-\xi)).$$ 

$ \mathcal{S} (\mathbb{R}^d)$ denotes the Schwartz space of infinitely smooth ($ C^\infty (\mathbb{R}^d)$) functions which, together with their derivatives, decay at infinity faster than any inverse polynomial. Its dual space of tempered distributions is denoted by $ \mathcal{S} ' (\mathbb{R}^d)$. 

Function $f$ belongs to  weighted space $L^p_v (\R^d)$ if $fv\in L^p(\R^d)$, $1\leq p \leq \infty$, where
$v$ is a weight function, see below.

\subsection{Weights} \label{weights}
Function $v$ is called a weight  function, or simly a weight, if it is locally bounded, non-negative, even and continuous on $\R^d$. Recall,
\begin{itemize}
\item[a)]  $v$ is submultiplicative if,
$v(x+y)\lesssim v(x)v(y)$, $x,y\in \R^d$,
\item[b)] $v $ is subconvolutive if $v^{-1}\in L^1$ and
 $(v^{-1}*v^{-1}) (x) \lesssim v^{-1}(x) $, $x\in \R^d$.
\item[c)] $m$ is moderate if there exists submultiplicative weight $v$ such that,
$$
m(x+y)\lesssim v(x)m(y), \qquad x,y\in \R^d.
$$
\end{itemize}

Clearly, if $m$ is submultiplicative than it is also moderate.
By ${\mathcal M}_{v}$ we denote the set of all $v-$ moderate weights.

\begin{ex} Standard examples of weight functions are given by
\be
\label{PrimerTezina}
m_{a, b, c, t}(x)=e^{a|x|^b}(1+|x|)^c(\log (e+|x|))^t , \quad a,c,t\in \mathbb R,\,\, b\geq 0,\,\, x\in \R^d.
\ee
\end{ex}

Properties of $m_{a, b, c, t}$ are given in the following Lemma from \cite{GRweights} (see also \cite{Fei}).

\begin{lema} If $m=m_{a, b, c, t}$ is given by \eqref{PrimerTezina} then
\label{GrochenigLema}
\begin{itemize}
\item[a)]  $m$ is submultiplicative if $a, c, t \geq 0$ and $0 \leq b \leq 1$,
\item[b)]  $m$ is subconvolutive if  $ a>0, c, t \in \mathbb{R}$ and $0<b<1$,
\item[c)]  $m$ is moderate if $a, c, t \in \mathbb{R}$ and $0 \leq b \leq 1$.

In addition, if $v$ is arbitrary submultiplicative weight then there exists $r>0$ such that 
$\dss v(x)\lesssim e^{ r |x|} $, $x\in \R^d$.
\end{itemize}
\end{lema}

In this paper we  consider weights of the form
\be
m_r^s(x)=e^{r| x|^{1 / s}},\quad r>0, \quad s>0\quad x\in \R^d.
\ee 
By Lemma \ref{GrochenigLema} it follows  that $m^s_{r}=m_{r,1/s,0,0}$ is submultiplicative and subconvolutive if $r>0$ and $s>1$.

Let us briefly discuss subconvolutivity of $m^s_{r}$, since it will be used later on. Clearly, $(m_r^s)^{-1}=m_{-r}^s \in L^1(\R^d)$. For the case $r>0$ and $0<s\leq 1$  \cite[Lemma 1.3.5]{CR}  gives
\be
\label{RodinoNejednakost}
m_{-r}^s*m_{-r}^s(x)\leq C m_{-2^{-1/s }r}^s (x),\quad x\in \R^d.
\ee 
Indeed, note that simple inequality $\dss |x+y|^{1/s}\leq 2^{\frac{1-s}{s}}  (|x|^{1/s}+|y|^{1/s}),\, x,y\in \R^d$ implies that
\be
m_{r}^s (x+y)\leq m_{2^{\frac{1-s}{s}}  r}^s(x) m_{ 2^{\frac{1-s}{s}}  r}^s(y),\quad x,y\in \R^d.\nonumber
\ee 
or equivalently
\be
m_{-r }^s (x-y) m_{-r }^s (y)  \leq m_{- 2^{\frac{s-1}{s}}  r}^s(x),\quad x,y\in \R^d.\nonumber
\ee 
Then we obtain
\begin{multline}
m_{-r}^s*m_{-r}^s(x)=\int_{\R^d}m_{-r}^s(x-y)m_{-r}^s(y) dy\\ \leq   \int_{\R^d} m_{-r/2}^s(x-y)\, m_{-r/2}^s(y)\,m_{-r/2}^s(y) dy \lesssim m_{- 2^{-1/s}  r}^s(x),\quad x \in \R^d.
\end{multline}

Function $f$ belong to a weighted space $L^p_v (\R^d)$ if $fv\in L^p(\R^d)$.









\subsection{STFT and modulation spaces} \label{stft-mod-spaces}

By $( f,\varphi )$ we denote  the dual pairing between  $ f\in \mathcal S' (\R^d)$ and $ \varphi \in \mathcal S (\R^d)$ and the dual pairing in the context of Gelfand-Shilov type spaces in Section \ref{gelfand-shilov}.

Let
$$
\pi(Z) g(t)=M_{\xi} T_x g=e^{2 \pi i t \cdot \xi} g(t-x),\;\;
g\in \mathcal{S}\left(\mathbb{R}^d\right)\backslash\{0\}, \;\;
Z=(x, \xi) \in \mathbb{R}^{2 d}.
$$

The short-time Fourier transform (STFT) of $f\in \mathcal{S}'\left(\mathbb{R}^d\right) $ with respect to a given 
window $g\in \mathcal{S}\left(\mathbb{R}^d\right)\backslash\{0\}$ is defined as
$$
V_g f(x, \xi)=\langle f, \pi(Z) g\rangle=\int_{\mathbb{R}^d} f(t) \overline{g(t-x)} e^{-2 \pi i t \xi} d t.
$$
The same formula over $ \mathbb{R}^{2 d}$ is given by
\begin{equation} \label{bigSTFT}
\mathcal{V}_{g} f(X, \Xi)=\int_{\mathbb{R}^{2d}} f(t_1,t_2) \overline{g((t_1,t_2)-(x_1,x_2))} e^{-2 \pi i (t_1,t_2)\cdot (\xi_1, \xi_2)} dt_1 dt_2,
\end{equation}
where $X=(x_1,x_2)$ and $\Xi=(\xi_1,\xi_2)$.

The (cross-)Wigner distribution is given by
\begin{equation} \label{wigner}
W(f, g)(x, \xi)=\int_{\mathbb{R}^{ d}} f(x-\frac{t}{2}) \overline{g(x+\frac{t}{2})} e^{-2 \pi i\xi t } d t, \qquad
f,g \in L^2 (\R^{d}),
\end{equation}
and when $g \in  \mathcal{S}\left(\mathbb{R}^d\right)$ it extends to 
$f \in  \mathcal{S}' \left(\mathbb{R}^d\right)$ by duality.

If $X=(x_1,x_2)$ and $\Xi=(\xi_1,\xi_2)$ belong to $\R^{2d}$, we write
\begin{multline*}
{\mathcal W} (f,g) (X, \Xi) \\ =\int_{\mathbb{R}^{ 2d}} f((x_1,x_2)-\frac{(t_1,t_2)}{2}) \overline{g((x_1,x_2)+\frac{(t_1,t_2)}{2})} e^{-2 \pi i (t_1,t_2) \cdot(\xi_1,\xi_2) } dt_1dt_2.
\end{multline*}

We recall that for $f\in \mathcal{S}'\left(\mathbb{R}^d\right) $ and $g\in \mathcal{S}\left(\mathbb{R}^d\right) $ it holds
$$W(f,g)(x,\xi)=2^d e^{4\pi i x \xi}\, V_{g^*}f(2x, 2\xi),\quad x, \xi\in \R^d, $$ where $g^*(x)=g(-x)$ 
(see \cite{Gr1}).



Modulation spaces are introduced by imposing mixed Lebesgue spaces norm to the STFT as follows.

Let $1\leq p,q\leq \infty$ and $m \in {\mathcal M}_{v}$, and let 
$g\in \mathcal{S}\left(\mathbb{R}^d\right)\backslash\{0\}$.
Then the weighted modulation space $M_m^{p, q}\left(\mathbb{R}^d\right)$ consists of $f\in \mathcal{S}'\left(\mathbb{R}^d\right)$ with the property
\begin{equation} \label{mod-space-norm}
\|f\|_{M_m^{p, q}}=\left(\int_{\mathbb{R}^d}\Big( \int_{\mathbb{R}^d}\left|V_g f(x, \xi)\right|^p \left| m(x, \xi)\right|^p d x\Big)^{q / p} d \xi\right)^{1 / q}<\infty,
\end{equation} 
with obvious changes if $p,q=\infty$. It is a Banach space with the norm $\|\cdot\|_{M_m^{p, q}}$, and it is well known that the definition does not depend on the choice of the window $g\in \mathcal{S}\left(\mathbb{R}^d\right)\backslash\{0\}$ in the sense that different Schwartz functions yield equivalent norms.

The original source for modulation spaces is \cite{F2}, see also \cite{Gr1}, and the recent monograph \cite{CR}.

\subsection{Pseudo-differential operators} \label{stft-mod-spaces}

We end this section with some remarks on the Weyl-H\"ormander pseudo-differential operators given by \eqref{Weyl-PsiDO}.
By a straightforward calculation one can show that for $ f \in \mathcal{S}' \left(\mathbb{R}^d\right)$ formula \eqref{Weyl-PsiDO} is (in the weak sense) equivalent to
$$
\left\langle a^w f, g\right\rangle=\langle a, W(g, f)\rangle \qquad  g \in \mathcal{S}\left(\mathbb{R}^d\right),
$$ 
where $W$ is the Wigner distribution given by \eqref{wigner}.
It is well known that the mapping 
$$a^w:\mathcal S(\mathbb R^d)\rightarrow \mathcal S'(\mathbb R^d)$$ 
is continuous.

The relation between $a^w $  and the STFT is given by the use of  the  symplectic structure on $\R^d$. 
This is an important observation when considering  metrics different than the Euclidean ones, see Section \ref{extension}.

\par

For $(\xi,\eta)\in \R^{2d}$ we denote $j(\xi,\eta)=(\eta, -\xi)$, Let $g\in \mathcal{S}\left(\mathbb{R}^d\right) \setminus \{ 0\}$. Recall that Lemma \cite[Lemma 14.5.1]{Gr1} implies the following  formula:
\be 
\left|\left( a^w \pi(X) g, \pi(Y) g\right)\right|=\left|\mathcal{V}_{\Phi} a\left(\frac{X+Y}{2}, j(Y-X)\right)\right|, \quad X,Y\in \R^{2d},
\ee where $\Phi = W(g,g)\in  \mathcal{S}\left(\mathbb{R}^{2d}\right)$. By the change of variables we obtain
\be
\left|\mathcal{V}_\Phi \sigma(U, V)\right|=\left|\left\langle\sigma^w \pi\left(U-\frac{j^{-1}(V)}{2}\right) g, \pi\left(U+\frac{j^{-1}(V)}{2}\right) g\right\rangle\right|,
\ee
$ U, V \in \mathbb{R}^{2 d} $.
Note that the standard symplectic form on $\R^{2d}$ is related to $j$ by
$$[(x,\xi), (y,\eta)]=\langle j(x,\xi), (y,\eta) \rangle = \langle \xi,y \rangle -\langle x,\eta \rangle,\quad x,y,\xi,\eta\in \R^d. 
$$ 
Also,
$$[(x,\xi), (y,\eta)]= \begin{bmatrix}
   y    & \eta \\
\end{bmatrix}\cdot \begin{bmatrix}
   0     & I\\
  -I      & 0
\end{bmatrix} \cdot \begin{bmatrix}
   x  \\
  \xi
\end{bmatrix}, \quad x,y,\xi,\eta\in \R^d,  
$$ 
where the standart symplectic matrix $ \dss \begin{bmatrix}
   0     & I\\
  -I      & 0
\end{bmatrix}$ is $2d\times 2d$ block matrix and $I$ is $d\times d$ identity matrix.

\section{Novel spaces for symbols of pseudo-differential operators} \label{spaces}

To extend results from \cite{PPHM}  into the framework of Gelfand--Shilov spaces given in \cite{CR} (see Section \ref{gelfand-shilov}) we need a careful preparation. This section contains original material, namely
the construction of particular Wiener-amalgam spaces $ W\left(L^{\infty},\mathcal{A}_r^s\right)$ which are used in the definition of new modulation spaces
$\widetilde{M}_{proj,s}^{\infty}$ and $ \widetilde{M}_{ind,s}^{\infty}$.
Then the symbols of pseudo-differential operators are distributions from 
$\widetilde{M}_{proj,s}^{\infty}$ or $ \widetilde{M}_{ind,s}^{\infty}$.

\subsection{Spaces of sequences} We introduce spaces of sequences which are convenient for our investigations. We note that 
these sequences are considered in \cite{Toft18} in the context of mapping properties of the Bargmann transforms
(cf. \cite[Definition 3.1]{Toft18}).

\par

Let $\Lambda$ be a discrete subgroup of $\R^d$, and 
\begin{equation} \label{our-weight}
m^s_r(\cdot)=e^{r|\cdot|^{1/s}}, \qquad r,s>0.
\end{equation}
Then the sequence $\mathbf{a}=(a_{\lambda})_{\lambda\in \Lambda}$ belongs to the Banach space  $l_{m^s_r}^{\infty}(\Lambda)$ if
$$||\mathbf{a}||_{r,s}=\|\mathbf{a}\, m^s_r \|_{\infty}=\sup_{\lambda\in \Lambda}|a_{\lambda}|e^{r |\lambda|^{1/s}}<\infty.$$
We will use the notation $\mathcal{A}_r^s=l_{m^s_r}^{\infty}(\Lambda)$.

It is well known that $l_v^{\infty} (\Lambda)$ is Banach algebra with respect to convolution if and only if $v$ is subconvolutive (see \cite{Fei}). Since  $m^s_r$ is not subconvolutive when $0<s\leq 1$  (see \eqref{RodinoNejednakost}), it follows that $\mathcal{A}_r^s$ is not a convolution algebra.

We collect some of the basic properties of  $\mathcal{A}_r^s$ in the following Lemma.

\begin{lema} Let $s,r>0$.
\label{PrvaLema}
\begin{itemize}

\item[i)] If $0<r_1<r_2$ then $\mathcal{A}_{r_2}^s \hookrightarrow \mathcal{A}_{r_1}^s \hookrightarrow l^1$ where $\hookrightarrow$ denotes continuous embedding. Moreover, this embedding is compact.
\item[ii)] $\mathcal{A}_r^s$ is involutive, i.e., if $\mathbf{a}\in \mathcal{A}_r^s$ then $\mathbf{a}^*=(a_{-\lambda})_{\lambda\in \Lambda}\in \mathcal{A}_r^s$ and $||\mathbf{a}^*||_{r,s} = ||\mathbf{a}||_{r,s}$.
\item[iii)] $\mathcal{A}_r^s$, is solid, i.e.,  if $ \mathbf{b}\in \mathcal{A}_r^s$ and $a_{\lambda}\leq b_{\lambda}$ for all $\lambda \in \Lambda$ then $ \mathbf{a}\in \mathcal{A}_r^s$ and  $||\mathbf{a}||_{r,s}\leq ||\mathbf{b}||_{r,s}$.

\item[iv)] Let  $\mathbf{a}, \mathbf{b}\in \mathcal{A}_r^s$ and $\mathbf{c}=\mathbf{a}*\mathbf{b}$. Then there exists $c_s=c\in(0,1]$ such that

 \be
\label{NormaNejednakost}
\|\mathbf {c}\|_{cr,s}\lesssim \|\mathbf {a}\|_{r,s}\|\mathbf {b}\|_{r,s} .
\ee

\end{itemize}
\end{lema}

\begin{proof}

We will only prove $i)$ and $iv)$ since $ii)$ and $iii)$ are straightforward.

$i)$ The compactness of the embedding is the consequence of the K\"othe theory of sequence spaces since the weights satisfy $e^{(r_1-r_2)|\lambda|^{1/s}}\to 0$ as $|\lambda| \to \infty$.

$iv)$ Note that Lemma \ref{GrochenigLema} and \eqref{RodinoNejednakost} implies that there exists $c\in(0,1]$ ($c=1$ for $s>1$ or $c=2^{-1/s}$ for $0<s\leq 1$) such that

\be
\label{Subconvolution}
\dss m_{-r}^s*m_{-r}^s (x)\lesssim  m_{-cr}^s (x) \quad x\in \R^d.
\ee
Therefore, if $\mathbf{c}=(c_{\lambda})_{\lambda\in \Lambda}$ we obtain
\begin{multline}
|\dss c_{\lambda}|\leq\sum_{\mu\in\Lambda}|a_{\lambda-\mu}||b_{\mu}|\leq \|\mathbf {a}\|_{r,s}\|\mathbf {b}\|_{r,s} (m_{-r}^s*m_{-r}^s(\lambda))\\
\lesssim \|\mathbf {a}\|_{r,s}\|\mathbf {b}\|_{r,s}  m_{-cr}^s (\lambda) ,\quad \lambda \in \Lambda,
\end{multline}
which implies $\dss  |\dss c_{\lambda}|\,m_{cr}^s (\lambda)\lesssim \|\mathbf {a}\|_{r,s}\|\mathbf {b}\|_{r,s}$, $\lambda\in \Lambda$. The claim follows by taking the supremum over $\lambda$.


\end{proof}

We introduce the Frech\'et space ($FS-$ space) and the dual Frechet space ($DFS-$ space) of sequences, respectively, by taking the projective and inductive limit topologies:
$$
\mathcal{A}_{proj}^s=\varprojlim _{r \to \infty} \mathcal{A}_r^s, \quad \mathcal{A}_{ind}^s=\varinjlim _{r \to 0} \mathcal{A}_r^s,\quad s>0.
$$ 
These spaces are  nuclear, and by Lemma \ref{PrvaLema} it follows that the spaces  $\mathcal{A}_{proj}^s$ and $\mathcal{A}_{ind}^s$ are closed under convolution for any $s>0$. Moreover, for a given  $r_0>0$ we have
$$
\mathcal{A}_{proj}^s\hookrightarrow \mathcal{A}_{r_0}^s\hookrightarrow  \mathcal{A}_{ind}^s \hookrightarrow l^1,\quad s>0.
$$ 


\subsection{Wiener amalgam spaces related to $\mathcal{A}_r^s$}
Next we introduce the Wiener-Amalgam space related to $\mathcal{A}_r^s$ when $r,s>0$.

Let $C$ be an open relatively compact set in $\R^{2d}$ that contains the origin and $\Lambda$ be a lattice such that $\R^{2d}\subset \dss \bigcup_{\lambda\in \Lambda}(\lambda+C)$, and $\Lambda$ be a lattice in $\R^{2d}$. For a locally bounded function $F$ on $\R^{2d}$ we set
\begin{equation}\label{win}
F_{\lambda}=\sup _{Y \in \lambda+C}|F(Y)|,\quad \lambda \in\Lambda.
\end{equation}
Then $F$ belongs to the Banach space $ W\left(L^{\infty},\mathcal{A}_r^s\right):=W^s_r$ if 
${\mathbf F}=\left(F_\lambda\right)_\lambda \in \mathcal{A}_r^s (\Lambda)$, and The norm 
in $W^s_r $ is given by $\|\displaystyle F\|_{W^s_r}=\|F_{\lambda}\|_{r,s}$. 
In particular, $W^s_r$ contains locally bounded functions whose decay rate at infinity is  bounded by
$m^s_{-r}(\cdot)=e^{-r |\cdot|^{1/s}}$.

\par

\begin{rem}
\label{RemarkWiener}
Note that $\dss W^s_r=W\left(L^{\infty},\mathcal{A}_r^s\right)=W(L^{\infty},l^{\infty}_{m^s_r}(\Lambda))=L_{m^s_r}^{\infty}(\R^{2d}). $
Moreover, by  Lemma \ref{PrvaLema} it follows that   $\dss  \mathcal{A}_r^s\hookrightarrow l^1$, and therefore $$W^s_r\hookrightarrow W(L^{\infty},{l^1}):= W(l^1),$$ 
where $W(l^1)$ is the Wiener space defined by $(F_\lambda)_{\lambda\in\Lambda}\in l^1$  (see (\ref{win})).

If $g\in W(l^1)$ and the lattice $\Lambda$ is sufficiently dense in $\R^{2d}$, 
then ${\mathcal G}(g,\Lambda)$ is a frame for $L^2 (\R^{d})$ (see \cite{Gr1} for details).
\end{rem}

By taking the projective and inductive limits we obtain
$$
W^{s}_{proj}= \varprojlim _{r \to \infty} W^s_r = W\left(L^{\infty}, \mathcal{A}_{proj}^s\right )\;\; {\rm and}
\;\; W^{s}_{ind}= \varinjlim _{r \to 0} W^s_r = W\left(L^{\infty}, \mathcal{A}_{ind}^s\right).
$$ 

We can prove the following Lemma.

\begin{lema}
Let $s>0$. $W^{s}_{proj}$ and $W^{s}_{ind}$ are involutive algebras with respect to convolution. In particular,
\be
\label{NormaWiener}
 \| F*G\|_{W^{s}_{c r}}\lesssim \| F\|_{W^s_{ r}} \| G\|_{W^s_{ r}}, \quad r>0.
\ee where $c\in (0,1]$ as in part $iv)$ of the Lemma \ref{PrvaLema}

\end{lema}

\begin{proof}
Directly form the definition it follows that $\dss W^s_r $  is an involutive algebra. 

Note that part $iv)$ of Lemma \ref{PrvaLema} implies that $\dss \mathcal{A}_r^s*\mathcal{A}_r^s\hookrightarrow \mathcal{A}_{c r}^s$. Then \cite[Theorem 3]{Fei1} implies that $W^{s}_{ r}*W^{s}_{ r}\hookrightarrow W^{s}_{c r}$ and the statement follows.
\end{proof}

\par 

\subsection{Modulation spaces related to $\mathcal{A}_r^s$}
We end the section by introducing modulation spaces with respect to  $\mathcal A^s_r$. For a symbol $a$ on $\R^{2d}$ let we define
$$
\mathcal G(a)(Y)=
\sup_{X\in\RR^{2d}}|\mathcal V_{\Phi}(a)(X,Y)|,\quad Y\in \R^{2d},$$
where  $\Phi = W(g,g), $ for a given window  $g\in\mathcal S(\RR^d) \setminus \{ 0\}.$

We say that $a\in \widetilde{M}^{\infty,\mathcal A^s_r}(\R^{2d}) $  if $G(a)\circ j \in W^s_{r}$, where $j(\xi,\eta)=(\eta, -\xi)$, $(\xi,\eta)\in \R^{2d}$. 

The space  $\widetilde{M}^{\infty,\mathcal A^s_r}(\R^{2d}) $ is a Banach space with the norm given by
$$\dss \|\sigma\|_{\widetilde{M}^{\infty,\mathcal A^s_r}}=\|G(\sigma)\circ j\|_{W^s_{r}}.$$ 
To shorten the notation we will write $ \dss \widetilde{M}^{\infty,\mathcal A^s_r}(\R^{2d}):=\widetilde{M}_{r,s}^{\infty}$.

Again, we consider the corresponding  projective and inductive limit spaces
$$\widetilde{M}_{proj,s}^{\infty}= \varprojlim _{r \to \infty} \widetilde{M}_{r,s}^{\infty}\quad \text{ and}\quad \widetilde{M}_{ind,s}^{\infty}= \varinjlim _{r \to 0} \widetilde{M}_{r,s}^{\infty}.$$

It can be proved that these 
spaces are related to the standard modulation spaces as given in \eqref{mod-space-norm} in the following way.

\begin{prop}

Let $g\in{\mathcal S}(\R^{d})\backslash\{0\}$, and $\Phi = W(g,g)$. Then

$$\dss \widetilde{M}_{proj,s}^{\infty}=\bigcap_{r>0} M_{1\otimes m^s_r}^{\infty,\infty}\quad { and}\quad \dss \widetilde{M}_{ind,s}^{\infty}=\bigcup_{r>0} M_{1\otimes m^s_r}^{\infty,\infty},$$ where

$$M_{1\otimes m^s_r}^{\infty,\infty}=\{f\in \mathcal{S}'(\R^{2d})\,|\,\sup_{X,\Xi\in \R^{2d} }|\mathcal{V}_\Phi  f(X,\Xi)|e^{r|\Xi|^{1/s}}<\infty \}.$$
\end{prop}

The class of symbols $\widetilde{M}^{\infty,\mathcal A^s_r}(\R^{2d}) $ is an appropriate choice when dealing with pseudo-differential operators on  Gelfand-Shilov type spaces with H\"ormander metrics.

\section{Approximate diagonalization in Gelfand-Shilov spaces} \label{gelfand-shilov}

In this section we first recall the definition of standard Gelfand-Shilov spaces, and then recall the results from 
\cite{CNR} and \cite{CR} which are necessary when extending results from Section \ref{extension} to weights that decays 
at infinity faster then the inverse of any polynomial. Basic facts of  Gelfand-Shilov spaces can be found in the original source \cite{GS}.


Let $s,t,A,B>0$. We start with the Banach space  $S_{t, B}^{s, A}\left(\mathbb{R}^d\right)$
  consisting of functions $f\in C^\infty(\RR^{d})$ with the finite norm
$$||f||_{\SSS^{s,A}_{t,B}}=\sup_{\alpha,\beta \in\NN^d_0, x\in\RR^d}\frac{\left|x^\alpha \partial^\beta f(x)\right|}{ A^{|\alpha|} B^{|\beta|} \alpha ! ^t \beta !^s}.
$$

Then, by taking projective and inductive limits we obtain
$$
\Sigma_t^s= \varprojlim _{A>0, B>0} S_{t, B}^{s, A} ; \quad S_t^s= \varinjlim _{A>0, B>0} S_{t, B}^{S, A} \text {. }
$$

Let us denote $S_{t,*}^{s}$ for $\Sigma_t^s$ or $S_t^s$. Also, the usual notation for isotropic spaces 
is $\SSS^s(\R^d)$ for $\SSS^s_s(\R^d), s\geq 1/2,$ and $\Sigma^s(\R^d)$ for $\Sigma^s_s(\R^d), s>1/2.$

The following results from \cite{CR} give rise to approximate diagonalization when considering the Euclidean metrics.

\begin{te} (\cite[Theorem 5.2.7]{CR}) \label{thm-prva}
Let  $s>0, m \in \mathcal{M}_v\left(\mathbb{R}^d\right), g \in M_{v \otimes 1}^1\left(\mathbb{R}^d\right) \backslash\{0\}$ such that
$$
\left\|\partial^\alpha g\right\|_{L_v^1\left(\mathbb{R}^d\right)} \lesssim C^{|\alpha|}(\alpha !)^s, \quad \alpha \in \mathbb{N}^d,
$$ 
for some $C>0$. If $f \in C^{\infty}\left(\mathbb{R}^d\right)$ the following conditions are equivalent:
\begin{itemize}
\item[(i)] There exists a constant $C>0$ such that
$$
\left|\partial^\alpha f(x)\right| \lesssim m(x) C^{|\alpha|}(\alpha !)^s, \quad x \in \mathbb{R}^d, \alpha \in \mathbb{N}^d .
$$
\item[(ii)] There exists a constant $\epsilon>0$ such that
$$
\left|V_g f(x, \xi)\right| \lesssim m(x) e^{-\epsilon|\xi|^{\frac{1}{s}}}, \quad x, \xi \in \mathbb{R}^{2 d}, \alpha \in \mathbb{N}^d .
$$
\end{itemize}
\end{te}


\begin{te}(\cite[Theorem 5.2.10]{CR}) \label{thm-druga}
 Let $s \geq 1 / 2$ and $m \in \mathcal{M}_v\left(\mathbb{R}^{2 d}\right)$.
\begin{itemize}
\item[$a)$] If $1 / 2 \leq s$ let $g \in S_s^s\left(\mathbb{R}^d\right)$.
\item[$b)$] If $1/2<s$ let $g \in \Sigma_s^s\left(\mathbb{R}^d\right)$.
\end{itemize}
Assume the following growth condition  on the weight $v$ :
$$
v(z) \lesssim e^{\epsilon|z|^{1 / s}}, \quad z \in \mathbb{R}^{2 d},
$$
for every $\epsilon>0$.

Let  $a \in C^{\infty}\left(\mathbb{R}^{2 d}\right)$. Then the following is equivalent:
\begin{itemize}
\item[(i)] The symbol $a$ satisfies
$$
\left|\partial^\alpha a(z)\right| \lesssim m(z) C^{|\alpha|}(\alpha !)^s, \quad z \in \mathbb{R}^{2 d},  \alpha \in \mathbb{N}^{2 d} .
$$
\item[(ii)] There exists $\epsilon>0$ such that
\begin{equation}
\label{ocenaAlmost}
\left|\left\langle a^{w} \pi(z) g, \pi(w) g\right\rangle\right| \lesssim m\left(\frac{w+z}{2}\right) e^{-\epsilon|w-z|^{\frac{1}{s}}}, \quad  z, w \in \mathbb{R}^{2 d} .
\end{equation}
\end{itemize}
\end{te}

\begin{te} (\cite[Theorem 5.2.12]{CR}) \label{thm-treca}
Let $s \geq 1 / 2$, $g \in S_s^s\left(\mathbb{R}^d\right)$,  $m \in \mathcal{M}_v\left(\mathbb{R}^{2 d}\right)$, and let $\mathcal{G}(g, \Lambda)$ be a Gabor superframe for $L^2\left(\mathbb{R}^d\right)$. If $a \in {C}^{\infty}\left(\mathbb{R}^{2 d}\right)$, thne the following properties are equivalent:
\begin{itemize}
\item[(i)]
 There exists $\epsilon>0$ such that the estimate \eqref{ocenaAlmost} holds.
\item[(ii)] There exists $\epsilon>0$ such that
$$
\left|\left\langle a^{w}\pi(\mu) g, \pi(\lambda) g\right\rangle\right| \lesssim m\left(\frac{\lambda+\mu}{2}\right) e^{-\epsilon|\lambda-\mu|^{\frac{1}{s}}}, \quad  \lambda, \mu \in \Lambda.
$$
\end{itemize}
\end{te}

Now we state a result on approximate diagonalization in the spirit of \cite{gro-rz}. The proof follows from a carefuli
inspection of the proofs of Theorems \ref{thm-prva}, \ref{thm-druga}, and \ref{thm-treca}, and will be given elsewhere.

Here below  $\dss  \mathcal{A}_{*}^s$ denotes $ \mathcal{A}_{proj}^s$ or  $\mathcal{A}_{ind}^s$,
$W^s_* $ stands for $W^{s}_{proj}$ or $W^{s}_{ind}$, and 
by $\dss \widetilde{M}_{*,s}^{\infty}$ we denote $\widetilde{M}_{proj,s}^{\infty}$ or $\widetilde{M}_{ind,s}^{\infty}$.

\begin{te} \label{GlavnaTeorema}  (\cite[Theorem 5.2.12]{CR})
Let $s\geq 1/2$  and $g\in {\mathcal S}_s ^{s}(\R^d)$. Then the following is equivalent:

\begin{itemize}
\item[a)] $a\in\widetilde M_{*,s}^{\infty}$,

\item[b)] there exists $H\in W^{s}_*$ such that
\begin{equation}\label{prvae}
|\langle a^w\pi(X)g,\pi(Y)g \rangle|\leq H(Y-X), \quad X,Y\in\RR^{2d},
\end{equation}

\item[c)] There exists a sequence $h\in\mathcal A_{*}^s (\Lambda)$ such that
$$
|\langle a^w\pi(\mu)g,\pi(\nu)\gamma \rangle|
\leq h(\nu-\mu), \quad \mu,\nu\in\Lambda.
$$
\end{itemize}
\end{te}

\section{Extensions on the spaces with the H\"ormander metric} \label{extension}

In this section we present some general results from \cite{PPHM} in the context of tempered distributions,
and symbol classes $S(M,g)$, see \eqref{cnbvlj135}. Extension of these results to Gelfand-Shilov spaces and their dual spaces of tempered ultra-distributions is a highly nontrivial task and will be the subject of our future investigations.

Before explaining the main results, we need to recall necessary notions related to the geometry in the observed spaces.
We refer to \cite{lernerB} for more details on the subject.

We assume that a Riemannian metric $g$ on $\mathbb R^{2d}$ is a Borel measurable section of the 2-covariant tensor bundle $T^2T^*\mathbb R^{2d}$ that is symmetric and positive-definite at every point. The corresponding quadratic forms are denoted  by the same symbol: $g_X(T):=g_X(T,T)$, $T\in T_X\mathbb R^{2d}$. For each $X\in \mathbb R^{2d}$, we  identify  $\mathbb R^{2d}$ with $T_X\mathbb R^{2d}$ that sends every $Y\in \mathbb R^{2d}$ to the directional derivative in direction $Y$ at $X$. Let $T\in \mathbb R^{2d}\backslash\{0\}$. We denote by $\partial_T$ the vector field on $\mathbb R^{2d}$ given by the directional derivative in direction $T$ at every point $X\in \mathbb R^{2d}$.  We denote by $\sigma:\mathbb R^{2d}\rightarrow \mathbb R^{2d}$ the isomorphism induced by the symplectic form; notice that ${}^t\sigma=-\sigma$, $\sigma_X  \in L(\R^{2d}, \R)$, $\sigma_X (T)=[X,T]$, $T\in \R^{2d}$. Let $g$ be a Riemannian metric on $\mathbb R^{2d}$ and, for $X\in \mathbb R^{2d}$, denote by $Q_X:\mathbb R^{2d}\rightarrow \mathbb R^{2d}$ the isomorphism induced by $g_X$. In particular, $Q_X$ is $2d\times 2d$ diagonal matrix with elements $q_X(E_j, E_j)$ where $\{E_j\}_{j=1}^{2d}$ is the basis of $\R^{2d}$.

For $X\in \mathbb R^{2d}$, set $Q^{\sigma}_X:={}^t\sigma Q_X^{-1}\sigma:\mathbb R^{2d}\rightarrow \mathbb R^{2d}$ and let $g^{\sigma}_X(T,S):=\langle Q^{\sigma}_XT,S\rangle$, $T,S\in \mathbb R^{2d}$. Then $g^{\sigma}$ is again a Riemannian metric on $\mathbb R^{2d}$ called the symplectic dual of $g$; it is also given by 
$$g^{\sigma}_X(T)=\sup_{S\in \mathbb R^{2d}\backslash\{0\}} [T,S]^2/g_X(S).$$ 

The Riemannian metric $g$ is said to be a H\"ormander metric \cite{hormander} (i.e., an admissible metric in the terminology of \cite{bony,lernerB}) if the following three conditions are satisfied:
\begin{itemize}
\item[$(i)$] (slow variation) there exist $C_0\geq 1$ and $r_0>0$ such that
    $$
    g_X(X-Y)\leq r^2_0\Rightarrow C^{-1}_0g_Y(T)\leq g_X(T)\leq C_0g_Y(T),
    $$
for all $ X,Y,T\in \mathbb R^{2d}$;
\item[$(ii)$] (temperance) there exist $C_0\geq 1$ and $N_0\geq 0$ such that
    $$
    \left(g_X(T)/g_Y(T)\right)^{\pm 1}\leq C_0(1+g^{\sigma}_X(X-Y))^N,\quad \mbox{for all}\,\, X,Y,T\in \mathbb R^{2d};
    $$
\item[$(iii)$] (the uncertainty principle) $g_X(T)\leq g^{\sigma}_X(T)$, for all $X,T\in \mathbb R^{2d}$.
\end{itemize}

Next we need an admissibility condition.
For a given metric $g$, a positive Borel measurable function $M$ on $\mathbb R^{2d}$ is said to be $g$-admissible if there are $C\geq 1$, $r>0$ and $N\geq 0$ such that
\begin{gather*}
g_X(X-Y)\leq r^2\Rightarrow C^{-1}M(Y)\leq M(X)\leq CM(Y), \quad \text{and}\\
\left(M(X)/M(Y)\right)^{\pm1}\leq C(1+g^{\sigma}_X(X-Y))^N,\quad \mbox{for all}\,\, X,Y\in \mathbb R^{2d}.
\end{gather*}
All the constants mentioned above are called \textit{admissibility constants}.
 
Given a $g$-admissible weight $M$, the space of symbols $S(M,g)$ is defined as the space of all $a\in\mathcal{C}^{\infty}(\mathbb R^{2d})$ for which
\beq\label{cnbvlj135}
\|a\|^{(k)}_{S(M,g)}=\sup_{l\leq k}\sup_{\substack{X\in \mathbb R^{2d}\\ T_1,\ldots, T_l\in \mathbb R^{2d}\backslash\{0\}}}\frac{|a^{(l)}(X;T_1,\ldots,T_l)|} {M(X)\prod_{j=1}^lg_X(T_j)^{1/2}}<\infty,\quad \forall k\in\NN.
\eeq

With this system of seminorms, $S(M,g)$ becomes a Fr\'echet space. One can always regularize the metric making it smooth without changing the notion of $g$-admissibility of a weight and the space $S(M,g)$; the same can be done for any $g$-admissible weight (see \cite{hormander}).

For any symbol $a\in\SSS(\mathbb R^{2d})$, we consider the Weyl quantization $a^w$, see \eqref{Weyl-PsiDO}
This correspondence extends to symbols in $\SSS'(\mathbb R^{2d})$ in a usual manner, and in 
this case $a^w:\SSS(\mathbb R^{2d})\rightarrow \SSS'(\mathbb R^{2d})$ is a continuous mapping. 
For any $a\in S(M,g)$ with  a $g$-admissible weight $M$, the operator $a^w$ is continuous on $\SSS(\mathbb R^{2d})$ and uniquely extends to an operator on $\SSS'(\mathbb R^{2d})$ (cf. \cite{hormander}).

\subsection{The symplectic short-time Fourier transform} We  extended in \cite{PP} the short time Fourier transform to the spaces of functions over the ground space equipped with the H\"ormader metrics with the aim to analyze the almost diagonalization of a class of pseudo differential operators that is,  the estimate of  the action of a pseudo differential operator on  wave packages accommodated to the involved metrics.

We denote by $\mathcal{F}_{\sigma}$ the symplectic Fourier transform on $\mathbb R^{2d}$:
$$
\mathcal{F}_{\sigma}f(X)=\int_{\mathbb R^{2d}} e^{-2\pi i [X,Y]} f(Y)dY,\quad f\in L^1(\mathbb R^{2d});
$$
recall that $\mathcal{F}_{\sigma}\mathcal{F}_{\sigma}=\operatorname{Id}$. ($i [X,Y]$ is the symplectic product) \\
\indent Let $\bold{\varphi}\in\SSS(\mathbb R^{2d})$ and set $\varphi_X:=\bold{\varphi}(X)$, $X\in \mathbb R^{2d}$. We define the \textit{symplectic short-time Fourier transform} $\VV_{\bold{\varphi}}f$ of $f\in \SSS'(\mathbb R^{2d})$ with respect to $\bold{\varphi}$ as
$$
\VV_{\bold{\varphi}}f(X,\Xi):=\mathcal{F}_{\sigma}(f\overline{\varphi_X})(\Xi)=\langle f,e^{-2\pi i [\Xi,\cdot]}\overline{\varphi_X}\rangle,\quad X,\Xi\in \mathbb R^{2d}.
$$
When $f\in L^1_{(1+|\cdot|)^{-s}}(\mathbb R^{2d})$ for some $s\geq 0$ ($|\cdot|$ is (any) norm on $\mathbb R^{2d}$), 
we have
$$
\VV_{\bold{\varphi}}f(X,\Xi)=\int_{\mathbb R^{2d}} e^{-2\pi i[\Xi,Y]} f(Y)\overline{\varphi_X(Y)} dY.
$$

The mapping $\Xi\mapsto e^{-2\pi i [\Xi,\cdot]}$,
from $\mathbb R^{2d}$ to the space of $\phi$ such that 
$\| \phi ^{(\alpha)} (\cdot)(1+|\cdot|)^{-1} \|_{L^{\infty}} < \infty$, for every $\alpha \in \mathbb{N}_0^d$, is well-defined and smooth. Therefore the mapping
$$(X,\Xi)\mapsto e^{-2\pi i[\Xi,\cdot]}\overline{\varphi}_X, 
\quad \text{from} \quad \mathbb R^{2d}\times \mathbb R^{2d}\rightarrow \SSS(\mathbb R^{2d}),$$
is strongly Borel measurable. Consequently, the function
$$(X,\Xi)\mapsto \VV_{\bold{\varphi}}f(X,\Xi) \quad \text{from} \quad
\mathbb R^{2d}\times \mathbb R^{2d}\rightarrow \CC,$$ 
is always Borel measurable, and if $\bold{\varphi}$ is of class $\mathcal{C}^k$, $0\leq k\leq \infty$, then $$\VV_{\bold{\varphi}}f\in\mathcal{C}^k(\mathbb R^{2d}\times \mathbb R^{2d}),  \quad 0\leq k\leq\infty.$$

For the study of $\VV_{\bold{\varphi}}$, we consider the Fr\'echet space
$$
\ds\lim_{\substack{\longleftarrow\\ s\rightarrow \infty}}L^{\infty}_{(1+|\cdot|)^s}(\mathbb R^{2d}\times \mathbb R^{2d}),
$$
and the inductive limit $(LB)$-spaces,
$$
\ds\lim_{\substack{\longrightarrow\\ s\rightarrow \infty}} L^{\infty}_{(1+|\cdot|)^{-s}}(\mathbb R^{2d}\times \mathbb R^{2d}), \quad \text{ and } \quad
\ds\lim_{\substack{\longrightarrow\\ s\rightarrow \infty}} L^1_{(1+|\cdot|)^{-s}}(\mathbb R^{2d}\times \mathbb R^{2d}).
$$

We note that $|\cdot|$ is any norm on $\mathbb R^{2d}\times \mathbb R^{2d}$ and  none of these spaces depends on the particular choice of $|\cdot|$. We have the following embeddings:
\begin{multline}\label{con-inc-den-for-pro-ind-lim-sps}
\SSS(\mathbb R^{2d}\times \mathbb R^{2d})\hookrightarrow \lim_{\substack{\longleftarrow\\ s\rightarrow \infty}} L^{\infty}_{(1+|\cdot|)^s}(\mathbb R^{2d}\times \mathbb R^{2d}) \hookrightarrow\lim_{\substack{\longrightarrow\\ s\rightarrow \infty}} L^{\infty}_{(1+|\cdot|)^{-s}}(\mathbb R^{2d}\times \mathbb R^{2d})\\
\hookrightarrow \lim_{\substack{\longrightarrow\\ s\rightarrow \infty}} L^1_{(1+|\cdot|)^{-s}}(\mathbb R^{2d}\times \mathbb R^{2d})\hookrightarrow \SSS'(\mathbb R^{2d}\times \mathbb R^{2d}).
\end{multline}


The following assertions are proved in  \cite{PP}:
\begin{itemize}
\item[$(i)$] The sesquilinear mapping
\begin{equation}\label{ses-lin-mapp-stft-for-s'}
\SSS'(\mathbb R^{2d})\times \SSS(\mathbb R^{2d})\rightarrow \lim_{\substack{\longrightarrow \\ s\rightarrow \infty}} L^{\infty}_{(1+|\cdot|)^{-s}}(\mathbb R^{2d}\times \mathbb R^{2d}),\quad (f,\bold{\varphi})\mapsto \VV_{\bold{\varphi}}f,
\end{equation}
is well-defined and hypocontinuous. Furthermore, for any bounded subset $B$ of $\SSS'(\mathbb R^{2d})$ there is $s>0$ such that $\VV_{\bold{\varphi}}f\in L^{\infty}_{(1+|\cdot|)^{-s}}(\mathbb R^{2d})$, for all $f\in B$, $\bold{\varphi}\in\SSS(\mathbb R^{2d})$, and the set of linear mappings
\begin{equation*}
\SSS(\mathbb R^d)\rightarrow L^{\infty}_{(1+|\cdot|)^{-s}}(\mathbb R^{2d}\times \mathbb R^{2d}),\quad\bold{\varphi}\mapsto \VV_{\overline{\bold{\varphi}}}f,\qquad f\in B,
\end{equation*}
is equicontinuous subset of $\mathcal{L}(\SSS(\mathbb R^{2d})),L^{\infty}_{(1+|\cdot|)^{-s}}(\mathbb R^{2d}\times \mathbb R^{2d}))$.
\item[$(ii)$] The sesquilinear mapping
\begin{equation}
\SSS(\mathbb R^{2d}\times \mathbb R^{2d})\rightarrow \lim_{\substack{\longleftarrow \\ s\rightarrow \infty}} L^{\infty}_{(1+|\cdot|)^s}(\mathbb R^{2d}\times \mathbb R^{2d}),\quad (\psi,\bold{\varphi})\mapsto \VV_{\bold{\varphi}}\psi,
\end{equation}
is well-defined and continuous.
\item[$(iii)$] Under a suitable (geometric) condition, and if $\bold{\varphi}\in\SSS(\mathbb R^{2d})$, then the conjugate-linear mapping $\SSS(\mathbb R^{2d})\rightarrow \SSS(\mathbb R^{2d}\times \mathbb R^{2d})$, $\psi\mapsto \VV_{\bold{\varphi}}\psi$, is well-defined and continuous.
\end{itemize}

Our main result in \cite{PP} is devoted to the almost diagonalization of $a\in S(M,g)$. We present a result which is a consequence of the main theorem in \cite{PP}.
Recall,  $\mathbb R^{2d}$ is occupied  with the  symplectic structure  $[(x,\xi),(y,\eta)]=\langle \xi,y\rangle-\langle \eta,x\rangle$.

\par

We denote by $\Psi^{g,L}_X$ the topological isomorphism over the Schwartz space ${\mathcal S}(\R^{2d})$,
$$
\Psi^{g}_X:\SSS(\mathbb R^{2d})\rightarrow \SSS(\mathbb R^{2d}),\;\; (\Psi^{g,L}_X\varphi)(Y)=\varphi({\widetilde{Q}_X}^{-1/2}Y),\, Y\in \mathbb R^{2d},
$$
and we extend it, by duality, to the topological isomorphism
\begin{equation}
\Psi^{g,L}_X:\SSS'(\mathbb R^{2d})\rightarrow \SSS'(\mathbb R^{2d}),\;\; \langle\Psi^{g}_Xf,\varphi\rangle=|\det \widetilde{Q}|^{1/2}\langle f,\varphi\circ\widetilde{Q}^{1/2}\rangle.
\end{equation}

In the simplest case  when the metric $g$ is symplectic ($g=g^\sigma$, see below),
 then our main   Theorem in \cite{PP} has the following  form: for each $N\in\NN$,
\begin{equation}\label{1equ-sym-ffr}
\mathbb R^{2d}\times \mathbb R^{2d}\ni(X,\Xi)\mapsto M((X+\Xi)/2)^{-1}(1+g_{\frac{X+\Xi}{2}}(X-\Xi))^N
\end{equation}
$$
\times \left|\left\langle \left(\Psi^{g}_{\frac{X+\Xi}{2}}a\right)^w \pi\left((Q_{\frac{X+\Xi}{2}})^{1/2}X\right)\chi,
\overline{\pi\left((Q_{\frac{X+\Xi}{2}})^{1/2}\Xi\right)\chi}\right\rangle\right| \in L^{\infty}(\mathbb R^{2d}\times \mathbb R^{2d}).
$$

This is a generalization of the result of Gr\"ochenig and Rzeszotnik \cite[Theorem 4.2 $(i)$-$(ii)$]{gro-rz}, since when $g$ is the Euclidean metric on $\RR^{2d}$ and $M(X)=1$ (which corresponds to the H\"ormander class $S^0_{0,0}$),  then
$\Psi^{g,L}_X=\operatorname{Id}$, $\tilde{Q}_X=\operatorname{Id}$, $\forall X\in \RR^{2d}$ and
$$
(X,\Xi)\mapsto \langle X-\Xi\rangle^N|\langle a^w \pi(X)\chi, \overline{\pi(\Xi)\chi}\rangle|\in L^{\infty}(\RR^{4n}).
$$

 The right hand side of \eqref{1equ-sym-ffr} can be explained by the use of simplectic and metaplectic transformations. Denote by $Mp(\mathbb R^{2d})$ and $Sp(\mathbb R^{2d})$ spaces of metaplectic operators and of symplectic transformations over $\mathbb R^{2d}$ (cf. \cite{deGosson}, \cite{lernerB}).  Note that thew interesting approach to the analysis of metaplectic transforms are given in \cite{CGNR} as well as in \cite{deGosson}. The so called lifting theorems using the H\"ormander metric are considered in \cite{ACT} and \cite{GT}.

Let $\Pi$ be the surjective homomorphism $\Pi:\operatorname{Mp}(\mathbb R^{2d})\rightarrow \operatorname{Sp}(\mathbb R^{2d})$. Since ${Q}^{1/2}_X\in\operatorname{Sp}(\mathbb R^{2d})$ for each $X\in \mathbb R^{2d}$, there exists $\Phi^{g}_X\in\operatorname{Mp}(\mathbb R^{2d})$ such that
$$
\Pi(\Phi^{g}_X)={Q}^{-1/2}_X\mbox \quad \text{and}\quad (\Psi^{g}_Xa)^w=(\Phi^{g}_X)^*a^w\Phi^{g}_X,\quad
a\in\SSS'(\mathbb R^{2d}), \;  X\in \mathbb R^{2d}.
$$

Let $\tau_X$, $X\in \mathbb R^{2d}$, be a metaplectic operator 
$$\tau_{(x,\xi)}\kappa=e^{2\pi i\langle y-x/2,\xi\rangle}\chi(\cdot-x), \quad \kappa(X,\Xi)\in \SSS(\mathbb R^{2d}).$$

Clearly, $\tau_{(x,\xi)}=e^{-\pi i\langle x,\xi\rangle}\pi(x,\xi)$. Then, since
$$\Omega\tau_X\Omega^*=\tau_{\Pi(\Omega)X} \mbox{ for all } \Omega\in\operatorname{Mp}(\mathbb R^{2d}), X\in \mathbb R^{2d},$$
(see \cite[Theorem 7.13, p. 205]{deGosson}) we infer that for each $X,Y\in \mathbb R^{2d}$, $\pi({Q}^{1/2}_Y X)$ is equal to $(\Phi^{g}_Y)^*\pi(X)\Phi^{g}_Y$ up to a constant of modulus $1$. Thus,  \eqref{1equ-sym-ffr} is equivalent to the following: for each $N\in \NN$,
\begin{multline*}
(X,\Xi)\mapsto M((X+\Xi)/2)^{-1}(1+g_{\frac{X+\Xi}{2}}(X-\Xi))^N \\
\times \left|\left\langle a^w \pi(X)\Phi^{g}_{\frac{X+\Xi}{2}}\chi, \overline{\pi(\Xi)\Phi^{g}_{\frac{X+\Xi}{2}}\chi}\right\rangle\right| \in L^{\infty}(W\times \mathbb R^{2d}).
\end{multline*}

\section*{Acknowledgements}
S. Pilipovi\'c is supported by the Serbian Academy of Sciences and Arts, project F10.
N. Teofanov and F.Tomi\'c were supported by the Science Fund of the Republic of
Serbia, $\#$GRANT No 2727, \emph{Global and local analysis of operators and
distributions} - GOALS. N. Teofanov gratefully acknowledges the financial support of the Ministry of Science, Technological Development and Innovation of the Republic of Serbia (Grants No. 451-03-66/2024-03/ 200125 $\& $ 451-03-65/2024-03/200125).

\end{document}